\documentclass[12pt, english, a4paper]{amsart}

\usepackage{graphicx}
\usepackage[utf8]{inputenc}
\usepackage[T1]{fontenc}
\usepackage{amsmath,amssymb, amsthm}
\usepackage[left=2.5cm, top=3cm, bottom=3cm, right=2.5cm]{geometry}
\usepackage{bm}

\def\mymathfont{\mathbf}

\def\(#1\){$(${\sl #1}\/$)$}

\newcommand{\R}{\mymathfont{R}}
\newcommand{\C}{\mymathfont{C}}
\newcommand{\Z}{\mymathfont{Z}}

\newcommand{\CP}{\mathop{\mymathfont{C}\mymathfont{P}}\nolimits}

\newcommand{\Cl}{\mathop{\mathrm{Cl}}\nolimits}
\newcommand{\Int}{\mathop{\mathrm{Int}}\nolimits}

\newcommand{\Bd}{\partial}

\newcommand{\M}{\mathop{\mathrm{Mod}}\nolimits}

\newcommand{\vspan}{\mathop{\mathrm{span}}\nolimits}
\newcommand{\grad}{\nabla}

\usepackage{euscript,color}

\def\green{\color[rgb]{0,.4,0}}
\def\blu{\color[rgb]{0.2,0,0.7}}

\def\bla{\color[rgb]{0,0,0}}

\newtheoremstyle{thm}{6pt plus 1pt minus 1pt}{6pt plus 1pt minus 1pt}{\slshape}{}{\scshape}{.}{5pt plus 1pt minus 1pt}{}
\newtheoremstyle{def}{6pt plus 1pt minus 1pt}{6pt plus 1pt minus 1pt}{}{}{\scshape}{.}{5pt plus 1pt minus 1pt}{}
\newtheoremstyle{rmk}{6pt plus 1pt minus 1pt}{6pt plus 1pt minus 1pt}{}{}{\scshape}{.}{5pt plus 1pt minus 1pt}{}
\newtheoremstyle{claim}{6pt plus 1pt minus 1pt}{6pt plus 1pt minus 1pt}{}{}{\slshape}{.}{5pt plus 1pt minus 1pt}{}

\theoremstyle{thm}
\newtheorem{newstatement}{newstatement}
\newtheorem{lemma}[newstatement]{Lemma}

\newtheorem*{theorem*}{Main Theorem}

\newtheorem{proposition}[newstatement]{Proposition}

\theoremstyle{def}
\newtheorem{definition}[newstatement]{Definition}

\theoremstyle{rmk}
\newtheorem{remark}[newstatement]{Remark}

\theoremstyle{claim}

\makeatletter
\renewcommand{\section}{\@startsection%
{section}
{1}
{0mm}
{1.5\bigskipamount}
{0.5\bigskipamount}
{\centering\normalsize\sc}}

\renewcommand{\subsection}{\@startsection%
{subsection}
{2}
{0mm}
{0.5\bigskipamount}
{0.5\bigskipamount}
{\normalsize\sc}}

\renewcommand{\paragraph}{\@startsection%
{paragraph}
{4}
{0mm}
{\bigskipamount}
{0pt}
{\normalsize\sl}}

\expandafter\let\expandafter\oldproof\csname\string\proof\endcsname
\let\oldendproof\endproof
\renewenvironment{proof}[1][\proofname]{%
  \oldproof[\slshape #1]%
}{\oldendproof}
\def\provedboxcontents#1{$\square$}

\let\phi\varphi
\title
[Holomorphic filling of $S^3$]
{A concave holomorphic filling\\[1pt]
of an overtwisted contact $\mathbf 3$-sphere}

\author{Naohiko Kasuya and Daniele Zuddas}
\address{Naohiko Kasuya, Department of Mathematics, Kyoto Sangyo University, Kamigamo-Motoyama, Kita-ku, Kyoto, 603-8555, Japan.}

\email{nkasuya@{\tt cc}.kyoto-su.ac.jp}

\address{Daniele Zuddas, Lehrstuhl Mathematik VIII, Mathematisches Institut der Universit\"{a}t Bayreuth, NW II, Universit\"{a}tsstr. 30, 95447 Bayreuth, Germany.}
\email{zuddas@uni-bayreuth.de}

\date{\today}

\keywords{Non-K\"{a}hler complex surface, overtwisted contact $3$-manifold, concave filling}
\subjclass[2010]{32V40, 32Q55, 57R17}

\begin{document}

\begin{abstract}
In this paper we prove that the closed $4$-ball admits non-K\"ahler complex structures with strictly pseudoconcave boundary. Moreover, the induced contact structure on the boundary $3$-sphere is overtwisted.
\end{abstract}

\maketitle

\section{Introduction}
In \cite{DKZ17}, Antonio J. Di Scala and we constructed a family of pairwise inequivalent complex surfaces 
$E = E(\rho_1, \rho_2)$ together with a holomorphic map $f \colon E \to \CP^1$ admitting compact fibers (the parameters $\rho_1$ and $\rho_2$ are such that $1 < \rho_2 < \rho_1^{-1}$). A relevant property of $E$ is that it is diffeomorphic to $\R^4$, hence providing an extension to real dimension four of a result of Calabi and Eckmann \cite{CE53}.

The compact fibers of $f$ were shown to be smooth elliptic curves and a singular rational curve with one node, and these are the only compact complex curves of $E$.
The existence of embedded compact holomorphic curves implies the non-existence of a compatible symplectic structure on $E$. Then, the complex surface $E$ is non-K\"ahler.

Further, in \cite{DKZ15b} we proved that $E$ cannot be realized as a complex domain in any smooth compact complex surface.

In the present paper, we study the structure of $E$ away from a compact subset, by providing an exhausting family of embedded strictly pseudoconcave 3-spheres (Lemma \ref{main thm1}). From this we derive our main theorem. In order to state our results, we recall the notion of {\sl Calabi-Eckmann type complex manifold} introduced in \cite{DKZ17}, which was inspired by the results of \cite{CE53}.

\begin{definition}
A complex manifold $M$ is said to be of Calabi-Eckmann type if there exist a compact complex manifold $X$ of positive dimension, and a holomorphic immersion $k \colon X \to M$ which is null-homotopic as a continuous map. 
\end{definition}

\begin{theorem*}
The closed ball $B^4$ admits a Calabi-Eckmann type complex structure $J$ with strictly pseudoconcave boundary. Moreover, the $($negative$)$ contact structure $\xi$ determined on $\Bd B^4 = S^3$ by the complex tangencies is overtwisted and homotopic as a plane field to the standard positive tight contact structure on $S^3$.
\end{theorem*}

In other words, $(B^4, J)$ is a concave holomorphic filling of the overtwisted contact sphere $(S^3, \xi)$. As far as we know, this is the first example of this sort in literature.

This 4-ball arises as a smooth submanifold of $E$ containing certain compact fibers of the map $f \colon E \to \CP^1$, and so it is evidently of Calabi-Eckmann type.

Our strategy for proving the theorem relies on finding a holomorphic open book decomposition embedded in $E$, whose pages are holomorphic annuli and whose monodromy is a left-handed (negative) Dehn twist of the annulus, hence the underlying 3-manifold $M \subset E$ is homeomorphic to $S^3$.

Moreover, we prove that $M$ can be smoothly approximated by a 1-parameter family of strictly pseudoconcave embedded 3-spheres $M_\tau \subset E$, for a suitable parameter $\tau \in (0,1)$. Namely, the complex domain outside the embedded 4-ball with corners $D \subset E$ bounded by $M$ is foliated by strictly pseudoconcave 3-spheres. This implies the existence of a strictly plurisubharmonic function on $E - D$.

As a consequence, the open book decomposition of $M$ is compatible with the contact structure of $M_{\tau}$ given by complex tangencies, which is then overtwisted. For the basics of the three-dimensional contact topology we use throughout the paper, the reader is referred, for example, to the book of Ozbagci and Stipsicz \cite[Chapters 4 and 9]{OZ04}.

\begin{remark}
By Eliashberg's classification of overtwisted contact structures on closed oriented $3$-manifolds \cite{El89}, the negative contact structure in the main theorem is uniquely determined up to isotopy.
\end{remark}

We point out that in all (odd) dimensions greater than three, a closed co-oriented overtwisted contact manifold (see \cite{BEM} for the definition) cannot be the strictly pseudoconcave boundary of a complex manifold. Indeed, such a holomorphic filling would give a strictly pseudoconvex CR structure on the contact manifold with reversed orientation. Then, it can be filled by a Stein space (Rossi's theorem \cite{Ro65}), and therefore it can be filled by a K\"ahler manifold (Hironaka's theorem \cite{Hi64}), which is impossible for an overtwisted contact manifold. In this sense, our example is peculiar to dimension three.

Lisca and Matic \cite[Theorem 3.2]{LM97} proved that any Stein filling $W$ of a contact 3-manifold can be realized as a domain in a smooth complex projective surface $S$. Then, $S - \Int W$ is a concave holomorphic filling of a Stein fillable contact 3-manifold.

On the other hand, Etnyre and Honda \cite{EH02} and Gay \cite{Ga02} proved that any closed co-oriented contact 3-manifold admits infinitely many pairwise inequivalent concave symplectic fillings.

Moreover, Eliashberg in \cite{El85} proved that for any closed contact $3$-manifold $(N, \xi)$, the 4-manifold $N \times [0,1]$ admits a complex structure such that the height function is strictly plurisubharmonic, providing a holomorphic cobordism of $(N, \xi)$ with itself. However, its proof is not constructive.

Our result gives a rather explicit complex cobordism of an overtwisted contact $3$-sphere with itself, by taking $\cup_{\tau \in [1/3,\, 1/2]} M_\tau \cong S^3 \times [0, 1]$ as a complex domain in $E$.

The paper is organized as follows.
In Section \ref{neighbor}, we recall the construction of the complex surface $E$ given in \cite{DKZ17} and present a holomorphic model of the complement $C = E - \Int D_1$, which will be helpful for the proof of the main theorem. In Section \ref{psh}, we construct a holomorphic open book decomposition embedded in $E$. Finally, in Section \ref{contact}, we prove the main theorem by showing the existence of a strictly plurisubharmonic function near the embedded open book decomposition based on contact topology.

\section{The complex surface $E$}~\label{neighbor}
In this section, we recall the construction of  $E$, by sketching the original one in \cite{DKZ17}. 
This will be helpful for the proof of our main theorem.

Throughout this paper we make use of the following notation: $\Delta(a, b) = \{z\in \C \mid a< |z| < b\}$, $\Delta[a, b] = \{z\in \C \mid a\leq |z| \leq b\}$, $\Delta(a) = \{z\in \C \mid |z| < a\}$ and similarly with mixed brackets. We also denote the closed disk and the circle of radius $a$ in $\C$, respectively, by $B^2(a)$ and $S^1(a)$. When $a = 1$, we drop it from the notation.

According to \cite{DKZ17}, the construction of $E = E(\rho_1, \rho_2)$ proceeds as follows. Let $\rho_1$ and $\rho_2$ be positive numbers such that $1 < \rho_2 < \rho_1^{-1}$, and choose $\rho_0$ such that $\rho_1 \rho_2^{-1} < \rho_0 < \rho_1$.

We want to realize $E$ as the union of two pieces. 
One of them is the product $V = \Delta (1, \rho_2)\times \Delta (\rho_0^{-1})$, and the other one is the total space $W$ of a genus-1 holomorphic Lefschetz fibration $h \colon W \to \Delta (\rho_1)$ with only one singular fiber $\Sigma$.

In order to define the analytical gluing between $V$ and $W$, we make use of the following Kodaira model \cite{Kod63}. 
Consider the elliptic fibration 
$$(\C^*\times \Delta (0, \rho_1))\slash \Z\to \Delta (0, \rho_1),$$ defined by the canonical projection on the quotient space of $\C^*\times \Delta (0, \rho_1)$ with respect to the $\Z$-action given by $n\cdot (w_1, w_2)=(w_1 w_2^n, w_2)$. 
Then, it canonically extends to a  singular elliptic fibration $h \colon W \to \Delta (\rho_1)$, and so we have an identification $W - \Sigma = (\C^* \times \Delta (0, \rho_1))\slash \Z$. The critical point of $h$ is non-degenerate, namely the complex Hessian is of maximal rank, and so $h$ is a genus-1 holomorphic Lefschetz fibration. 
In the following we shall keep the convention of denoting by $(w_1, w_2)$ the usual complex coordinates of $\C^* \times \Delta(\rho_1) \subset \C^2$ when they are referred to $W$ (up to the above identification), and by $(z_1, z_2)$ the usual coordinates of $\C^2$ when they are referred to $V \subset \C^2$.

Now, let us consider the multi-valued holomorphic function $\phi \colon \Delta (0, \rho_1)\to \C ^{\ast }$ defined by 
\begin{eqnarray*}
\phi(w)=\exp \left(\frac{1}{4\pi i}(\log w)^2 - \frac{1}{2} \log w \right). 
\end{eqnarray*}
We denote by $\Phi \colon U \to W$ the holomorphic map defined by 
$$\Phi(z_1, z_2) = [(z_1\, \phi(z_2^{-1}),\, z_2^{-1})],$$ 
where $U \subset \C^* \times \Delta(\rho_1^{-1}, \rho_0^{-1})$ is a certain open subset that will be specified later. Notice that $\Phi$ is single-valued. This depends on the fact that any two branches $\phi_1$ and $\phi_2$ of $\phi$ are related by the formula $\phi_2(w) = w^k \phi_1(w)$ for some $k \in \Z$, which is compatible with the above $\Z$-action. For the purposes of this Section, we take $U = \Delta(1,\rho_2) \times \Delta(\rho_1^{-1}, \rho_0^{-1}) \subset V$.

It follows that $\Phi$ is a biholomorphism between $U \subset V$ and $\Phi(U) \subset W$.

We are now ready to holomorphically glue $V$ and $W$ by identifying the open subsets $U \subset V$ and $\Phi (U) \subset W$ by means of $\Phi$. That is, we define the complex surface $$E = E(\rho_1, \rho_2) = V \cup_{\Phi} W.$$
We consider $V$ and $W$ as open subsets of $E$ via the quotient map. 

By construction, there is a holomorphic map $f \colon E\to \CP^1$ defined by the canonical projection onto the second factor on $V$ and by the elliptic fibaration $h$ on $W$, where $\CP^1$ is regarded as the result of gluing the disks $\Delta(\rho_0^{-1})$ and $\Delta(\rho_1)$ by identifying $\Delta(\rho_1^{-1}, \rho_0^{-1}) \subset \Delta(\rho_0^{-1})$ with $\Delta(\rho_0, \rho_1) \subset \Delta(\rho_1)$ by means of the inversion map $z \mapsto z^{-1}$. 

Notice that the resulting complex surface $E$ does not depend on $\rho_0$, since this parameter determines only the size of the gluing region.

\medskip

\begin{remark}
By taking $\rho'_1$ and $\rho'_2$ such that $\rho_2 < \rho'_2 < (\rho'_1)^{-1} < \rho_1^{-1}$, our construction yields an obvious holomorphic embedding of $E$ in $E' = E(\rho_1', \rho_2')$ as a relatively compact complex domain. The closure $\hat E =\Cl E$ in $E'$ has Levi flat piecewise smooth boundary, and $\Bd \hat E$ is homeomorphic to $S^3$. This agrees with the interpretation of the map $f \colon E \to \CP^1$ we gave in \cite{DKZ17} as the restriction of the Matsumoto-Fukaya fibration $S^4 \to S^2$ \cite{Mat82} to the complement of a neighborhood of the negative critical point in $S^4$. This also relates to the embedded open book decomposition that we construct in Proposition \ref{main thm2}.
\end{remark}


Let $V' = \Delta (1, s)\times \Delta (\rho_1^{-1}, \rho _0^{-1})$, 
where the additional parameter $s$ is chosen so that   
$\rho_0^{-1} < s < \rho_1^{-1} \rho_2$.
Let $U'$ be the subset of $V'$ defined by 
$U' = \{ (z_1, z_2)\in V' \mid |z_2|<|z_1| \}$.
We put $V'' = V \cup V' \subset \C^2$ and identify a point $(z_1, z_2) \in U'$ with $\psi(z_1,z_2)$, where
$\psi \colon U' \to V'$ is the holomorphic embedding defined by $\psi(z_1, z_2) = (z_1z_2^{-1}, z_2)$.
Let $Y = V''/\!\!\sim$ be the quotient.

\begin{proposition}\label{model}
The manifold $Y=V''/\!\! \sim $ is biholomorphic to the preimage of the disk $\Delta (\rho_0^{-1})\subset \CP^1$ by the holomorphic fibration $f \colon E\to \CP^1$. 
\end{proposition}

\begin{proof}
The preimage $f^{-1}(\Delta (\rho_0^{-1}))$ is described as follows. 
Let $W(\rho_0, \rho_1)$ be the subset of $W$ given, in the Kodaira model above, by 
$$W(\rho_0, \rho_1)=(\C^{\ast }\times \Delta (\rho_0, \rho_1))\slash \Z = f^{-1}(\Delta (\rho_0, \rho_1)),$$ being $f = h$ in $W(\rho_0, \rho_1)$.
Then, we have $U' \subset W(\rho_0, \rho_1)$, and so  
$$f^{-1}(\Delta (\rho_0^{-1}))= V \cup_{U \sim U'} W(\rho_0, \rho_1).$$

Now, we define a map $\Psi \colon Y\to f^{-1}(\Delta (\rho_0^{-1}))$ by putting $\Psi([(z_1,z_2)])=(z_1, z_2)$ on $V/ \!\!\sim$ and
$\Psi([(z_1,z_2)]) = \Phi(z_1,z_2)$ on $V'/\!\!\sim$.
It is easy to check that $\Psi$ is well-defined and it is a biholomorphism.
\end{proof}

In order to obtain the complement $C \subset E$ of a 4-ball $D$ containing the singular fiber of $f$,
we remove from $Y$ the subset $$Z =  \{ (z_1,z_2) \mid c_1<|z_1|<c_2 \} \subset V,$$ where  $s\rho_1<c_1<c_2<\rho_2$. 
Then, by Proposition \ref{model}, it is enough to set $C = Y - Z$.

\section{The holomorphic open book decomposition}\label{psh}
We briefly recall the notion of open book decomposition of a 3-manifold. For a more thorough treatment, the reader is referred to Ozbagci and Stipsicz \cite[Chapter 9]{OZ04} and to Rolfsen \cite[Chapter 10K]{Ro90}. 

By an {\sl open book decomposition} of a closed, connected, oriented, manifold $M$ of real dimension three, we mean a smooth map $f \colon M \to B^2$ such that the followings hold:

\begin{enumerate}
\item the restriction $f_| \colon \Cl(f^{-1}(\Int B^2)) \to B^2$ is a (trivial) fiber bundle with fiber a link $L = f^{-1}(0)$, called the {\sl binding} of the open book; 
\item the map $\phi\colon M - L \to S^1 = \Bd B^2$ defined by $\phi(x) = f(x)/|f(x)|$ is a fiber bundle.
\end{enumerate}
The closure of every fiber $F_\theta = \Cl(\phi^{-1}(\theta))$, for $\theta \in S^1$, is a compact surface in $M$, called a {\sl page} of the open book, and $\Bd F_\theta = L$.
By a little abuse of terminology, we call pages of $f$ also the surfaces $f^{-1}(\theta)$, for all $\theta \in S^1 = \Bd B^2$. The two kinds of pages are ambient isotopic in $M$ to each other.

Given an open book decomposition $f \colon M \to B^2$, the orientations of $M$ and of $B^2$ induce an orientation on the pages, and hence on the binding $L = \Bd F_\theta$.

To an open book decomposition $f \colon M \to B^2$ it is associated the {\sl monodromy} $\omega_f$ of the bundle $\phi$, which is a diffeomorphism of a page $F_*$ that fixes the boundary pointwise, and it is well defined up to isotopy fixing the boundary.

On the other hand, given an element $\omega$ of the mapping class group $\M_{g,b}$ of a compact, connected, oriented surface $F_{g,b}$ of genus $g\geq0$ and with $b\geq 1$ boundary components, there is an open book decomposition $f_\omega \colon M_\omega \to B^2$ with monodromy $\omega$ and page $F = F_{g,b}$, and this is uniquely determined up to orientation-preserving diffeomorphisms. The construction goes as follows. Take a representative $\psi \colon F \to F$ of the isotopy class $\omega$, and consider the mapping torus $T_\omega = (F \times \R)/\Z$, where the $\Z$-action is generated by the diffeomorphism $\tau \colon F \times \R \to F \times \R$ defined by $\tau(x, t) = (\psi(x), t-1)$.

Let $M_\omega$ be the result of gluing $\Bd F \times B^2$ to $T_\omega$ along the boundary, by means of the obvious identifications $\Bd (\Bd F \times B^2) \cong \Bd F \times S^1 \cong \Bd F \times (\R/\Z) \cong \Bd T_\omega$, where the last identification comes from the fact that $\psi$ is the identity on $\Bd F$. Then, let $f \colon M_\omega \to B^2$ be the canonical projection $\Bd F \times B^2 \to B^2$ on $\Bd F \times B^2 \subset M_\omega$, while it is the projection $T_\omega \to \R /\Z \cong \Bd B^2$ on $T_\omega \subset M_\omega$.

Consider an oriented surface $F$ and let $\gamma \subset \Int F$ be a connected simple curve. A {\sl Dehn twist} $\delta_\gamma \colon F \to F$ about the curve $\gamma$ is a diffeomorphism of $F$ such that away from a tubular neighborhood $T$ of $\gamma$ in $F$, $\delta_\gamma$ is the identity, while in $T \cong S^1 \times [0,1]$ the diffeomorphism $\delta_\gamma$ either corresponds to the map $\delta_- \colon S^1 \times [0,1] \to S^1 \times [0,1]$ defined by $$\delta_-(z, t) = (z\, e^{-2\pi\, i\, t},\, t),$$ or to the map $\delta_+ = \delta_-^{-1}$, where $S^1 \times [0,1]$ is endowed with the product orientation and its identification with $T \subset F$ is orientation-preserving. In the former case, $\delta_\gamma$ is called a {\sl left-handed (or negative) Dehn twist}, while in the latter it is called a {\sl right-handed (or positive) Dehn twist}. By changing the orientation of $F$, the two types of Dehn twists are swapped.

The 3-sphere admits an open book decomposition $h_- \colon S^3 \to B^2$ with binding the negative Hopf link $H_-$, and with page the annulus $S^1 \times [0,1]$. The monodromy is the left-handed Dehn twist about the core circle $\gamma = S^1 \times \left\{\frac{1}{2}\right\}$ of the annulus (there is also the positive version $h_+ \colon S^3 \to B^2$ of this). This is the well-known realization of the (negative) Hopf link in $S^3$ as a fibered link, with page the Hopf band \cite{Ha82}.

The following proposition will be helpful in the proof of Main Theorem. We keep the notations of Section \ref{neighbor}.

\begin{proposition}\label{main thm2}
There is a piecewise smooth embedded 3-sphere $M \subset E$ such that the restriction $f_{|M} \colon M \to B^2$ of the holomorphic map $f \colon E \to \CP^1$, is diffeomorphic to the open book decomposition $h_-$ of $S^3$ described above, with $B^2$ a suitable closed disk in $\Delta(\rho_0^{-1}) \subset \CP^1$. Every page of $f_{|M}$ is a holomorphic annulus in an elliptic fiber of $f$. Moreover,  $M$ is not globally smooth, since it has corners along the two linked tori given by $\Bd f_{|M}^{-1}(\Bd B^2)$, on the complement of which $M$ is foliated by holomorphic curves. Thus, $M$ is Levi flat in $E$.
\end{proposition}

We endow $M$ with the orientation determined by the open book decomposition, where the pages are oriented by the induced complex structure, and the base disk $B^2$ takes the orientation from $\CP^1$. By construction, this disk is in the part of $\CP^1$ that corresponds, via the map $f$, to the Stein open subset $V \subset E$, with the boundary in the gluing region.

Fix two numbers $c$ and $\epsilon$ such that $\rho_0 < c < \rho_1$ and $0 < \epsilon < (\rho_0 - \rho_1 \rho_2^{-1})/2$. We put $a = \rho_2 - \epsilon$ and $b = c^{-1} + \epsilon$, and let $A = \Delta[a, b]$. It is then straightforward to check that  $(\lambda^k A) \cap A = \emptyset$ for all $\lambda \in \Delta[c, \rho_1]$ and for all $k \in \Z$.

\begin{proof}
Consider the set $$G = f^{-1}(S^1(c)) - \Phi(\Delta(bc, a) \times S^1(c^{-1})) \subset E,$$ with $S^1(c) \subset \Delta(\rho_1) \subset \CP^1$. The map $f_G = f_{|G} \colon G \to S^1(c) \cong S^1$ is a compact annulus bundle over the circle $S^1(c) \subset \Delta(\rho_1) \subset \CP^1$ of radius $c$. Here $S^1(c)$ has the clockwise orientation in the disk $\Delta(\rho_1)$, namely it is oriented as the boundary of the disk it bounds in $\Delta(\rho_0^{-1}) \subset \CP^1$. This choice depends on the inversion in the map $\Phi'$ below.

This bundle is trivial, and a trivialization is provided by the map $\Phi' \colon A \times S^1 \to G$ that is defined by $$\Phi'(w_1, w_2) = \Phi(w_1, c^{-1} w_2) = [(w_1\, \phi(c\, w_2^{-1}),\, c\, w_2^{-1})].$$ Notice that $\Phi'$ is holomorphic on every fiber.

Now, we construct an open book decomposition of $S^3$ embedded in $E$. We begin with an abstract description of this open book, and then we see how it is embedded in $E$.

Let $\psi_1$ be the identity map of $S^1(a) \times S^1$, and let $$\psi_2 \colon S^1(b) \times S^1 \to S^1(b) \times S^1$$ be defined by $\psi_2(w_1, w_2) = (w_1 w_2, w_2)$.

We use the diffeomorphism $\psi = \psi_1 \cup \psi_2 \colon \Bd(\Bd A \times B^2) \to \Bd(A \times S^1)$ to construct the oriented 3-manifold $$M = (\Bd A \times B^2) \cup_\psi (A \times S^1)$$ obtained by gluing $\Bd A \times B^2$ to $A \times S^1$ along the boundary (these two pieces are oriented in the canonical way).

Let $p \colon M \to B^2$ be defined by
$p(w_1, w_2) = w_2$, for $(w_1, w_2) \in \Bd A \times B^2$ or $(w_1, w_2) \in A \times S^1$.
It is clear that $(M, p)$ is an open book decomposition of $M$ with binding $L = \Bd A \times \{0\} \subset \Bd A \times B^2 \subset M$ and the annulus $A$ as the page.

Now, we show that the monodromy of $p$ is the diffeomorphism
$\delta \colon A \to A$ defined by $$\delta(z) = z\, e^{2\pi\, i\, \tau(|z|)},$$ where $\tau \colon [a, b] \to [0,1]$ is an increasing diffeomorphism (for example, the affine one). Thus, $\delta$ is the identity on $\Bd A$.
Let $$T(\delta) = \frac{A \times [0,1]}{(z,1) \sim (\delta(z), 0)}$$ be the mapping torus of $\delta$.

The open book decomposition with page $A$ and monodromy $\delta$ represents a 3-manifold $B(\delta)$ obtained by capping off $T(\delta)$ with $\Bd A \times B^2$ glued along the boundary by the identity, up to the obvious identification $\Bd B^2 = S^1 \cong [0,1]/(0\sim 1)$.

Define the map $k \colon T(\delta) \to A \times S^1$ by setting $$k([(z,t)]) = (z\, e^{2\pi\, i\, \tau(|z|)\, (t-1)},\, e^{2\pi\, i\, t}).$$ Then, $k$ is an orientation-preserving fibered diffeomorphism. 

The gluing maps $\psi_1$ and $\psi_2$ used for building $M$ correspond, by means of $k$, to the identity of $\Bd (T(\delta)) = \Bd A \times S^1$. This implies that there is a diffeomorphism $M \cong B(\delta)$, with respect to which the open book $p$ corresponds to that of $B(\delta)$, and so $\delta$ is the monodromy of $p$.

In order to understand $\delta$, we consider the diffeomorphism $q \colon A \to S^1 \times [0,1]$ defined by $$q(z) = (\bar z / |z|,\, \tau(|z|)).$$ This is orientation-preserving, as it can be easily shown by writing $q$ in polar coordinates. Moreover, $q^{-1}(w, t) = \tau^{-1}(t)\, \bar w$.

It is now straightforward to prove the identity $\delta_- = q \circ \delta\circ q^{-1}$, where $\delta_-$ is the left-handed Dehn twist defined above.
Therefore, $\delta$ is a left-handed Dehn twist of $A$ about the curve $\gamma \subset A$ of equation $\tau(|z|) = 1/2$ (that is, the core of $A$). It follows that $p \colon M \to B^2$ is equivalent to the open book $h_-$ of $S^3$, and in particular $M \cong S^3$.

Next, we define an embedding $g \colon M \to E$ in the following way
$$g(z_1, z_2) =
	\begin{cases}
		\Phi'(z_1, z_2), & \text{for } (z_1, z_2) \in A \times S^1\\
		j(z_1, c^{-1} z_2), & \text{for } (z_1, z_2) \in S^1(a) \times B^2\\
		j(c z_1, c^{-1} z_2), & \text{for } (z_1, z_2) \in S^1(b) \times B^2,
	\end{cases}$$
where $j \colon V \hookrightarrow E$ is the inclusion map.

We show that $g$ is well-defined. For $(z_1, z_2) \in S^1(a) \times S^1$, we have $$g(z_1, z_2) = j(z_1, c^{-1} z_2) = \Phi(z_1, c^{-1} z_2) = [(z_1\, \phi(c\, z_2^{-1}),\, c\, z_2^{-1})] = (\Phi'\circ \psi_1)(z_1, z_2).$$
Finally, we check consistency at $(z_1, z_2) \in S^1(b) \times S^1$. We have
\begin{equation*}
\begin{split}
	g(z_1, z_2) & = j(c\, z_1, c^{-1} z_2) = \Phi(c\, z_1, c^{-1} z_2) = [(c\, z_1\, \phi(c\, z_2^{-1}), c\, z_2^{-1})]= \\
	& = [(z_1 z_2\, \phi(c\, z_2^{-1}), c\, z_2^{-1})] = (\Phi'\circ \psi_2)(z_1, z_2),
\end{split}
\end{equation*}
where we are using the $\Z$-action considered in Section \ref{neighbor}.

By abusing notation, we still denote by $M \subset E$ the image of $g$.
Therefore, $M$ is a piecewise smooth embedded submanifold of $E$, although it is not globally smooth. Indeed, the two codimension-0 submanifolds of $E \cong \R^4$ bounded by $M$ have corners along $\Bd A \times S^1 \subset M$. Away from the corners, $M$ is foliated by holomorphic curves, hence it is Levi flat. These holomorphic curves are the images of the disks $\{z_1\} \times B^2$ and the images of the annuli $A \times \{z_2\}$ by the embedding $g$, with $(z_1, z_2) \in \Bd A \times S^1$.
\end{proof}

Let $D \subset E$ be the compact submanifold bounded by $M$, and let $C$ be the non-compact one. Hence, $E = D \cup_M C$.

The argument based on Kirby calculus in \cite{DKZ17} proves the following proposition.
\begin{proposition}
Up to smoothing the corners, $D$ is diffeomorphic to $B^4$ and $C$ is diffeomorphic to $S^3 \times (0, 1]$.
\end{proposition}

The same conclusion follows from the existence of a proper continuous function $u \colon C \to (0,1]$, which is smooth, regular and strictly plurisubharmonic in $\Int C$. In the next section, we show the existence of such a function to prove Main Theorem.

\section{The key Lemma and the proof of the Main Theorem}\label{contact}
In this section we prove the following lemma, which is needed for the proof of our Main Theorem.

\begin{lemma}\label{main thm1}
There exists a smooth 3-sphere $M_1 \subset E$ such that: \(1\) the non-compact submanifold $C_1 \subset E \cong \R^4$ bounded by $M_1$, admits a proper smooth regular strictly plurisubharmonic function $u \colon C_1 \to (0, 1]$, and \(2\) the complement $E - C_1$ is of Calabi-Eckmann type.
\end{lemma}

\begin{remark}
Property \(1\) of the Lemma implies that $M_1$ is smoothly standard in $E$, meaning that there exists a diffeomorphism $E \to \R^4$ mapping $M_1$ to the standard unit sphere $S^3$. Therefore, $C_1 \cong S^3\times (0, 1]$ and $D_1 = E - \Int C_1 \cong B^4$.
\end{remark}

Lemma \ref{main thm1} follows from the construction of $C$ in Section \ref{neighbor} and the following lemma. 
We use the notations  
$A_k=\left\{(r_k, \theta _k) \mid 1<r_k<a_k \right\}$, for $k=1, 2$, where $(r_k, \theta _k)$ denote the polar coordinates on $\C $. 
Let $J$ be the standard complex structure on the product $A_1\times A_2$.

\begin{lemma}\label{V_2}
Let $p \colon (1,a_1)\to (1, a_2)$ be a smooth function. 
Then, the hypersurface $S= \{(z_1,z_2) \mid |z_2|=p(|z_1|) \}$ of $A_1\times A_2$ is diffeomorphic to $T^2\times (0,1)$. Moreover, the complex tangencies define the positive (resp. negative) standard contact structure on $S$ if and only if the function $\log p(e^x)$ is strictly concave (resp. strictly convex). 
\end{lemma}

\begin{proof}
It is obvious that the following map is a diffeomorphism between $S$ and $T^2\times (1, a_1)$: 
$$S\to T^2\times (1, a_1);\quad (z_1,z_2)\mapsto (\theta _1, \theta _2, r_1).$$
Now, we describe the complex tangency $\xi =TS\cap JTS$. 
Since the two vector fields 
\begin{eqnarray*}
v=\frac{\partial}{\partial r_1} + p'(r_1) \frac{\partial}{\partial r_2} \quad \text{and} \quad
Jv=\displaystyle \frac{1}{r_1} \frac{\partial}{\partial \theta _1} +\frac{p'(r_1)}{p(r_1)} \frac{\partial}{\partial \theta _2}
\end{eqnarray*}
are tangent to the submanifold $S$, the complex tangency is given by $\xi =\vspan\{v, Jv\}$. 
Hence, $\xi $ is described as the kernel of the $1$-from $\alpha $ on $S\cong T^2\times (1, a_1)$ given by $$\alpha = r_1 p'(r_1)\, d\theta _1-p(r_1)\, d\theta _2.$$
This $1$-form gives a linear characteristic foliation with slope $\frac{r_1p'(r_1)}{p(r_1)}$ on the pre-Lagrangian torus $\left\{r_1=\text{const} \right\}$.
The property of being a contact form is equivalent to the strictly concavity (strictly convexity) of $\log p(e^x)$. 
Indeed, the sign of the second derivative  
$$(\log p(e^x))''= \frac{e^x((p'(e^x)+e^{x} p''(e^x)) p(e^x) - e^{x} p'(e^x)^2)}{p(e^x)^2}$$
corresponds to that of the derivative of the slope of the characteristic foliation 
$$\left(\frac{r_1p'(r_1)}{p(r_1)}\right)' = \frac{(p'(r_1) + r_1 p''(r_1))p(r_1) - r_1 p'(r_1)^2}{p(r_1)^2}.$$ 
This completes the proof.  
\end{proof}

We note that the graph of the function $y = \log {p(e^x)}=({\exp ^{-1}}\circ p \circ \exp ) (x)$ is equal to 
the subset $\left\{(\log X, \log Y) \mid Y=p(X)  \right\}\subset \R^2$. 
Notice also that Lemma~\ref{V_2} is a special case of the following well-known fact. 

\begin{proposition}
Let $\Omega =\left\{ (z_1,z_2) \mid |z_2|<\exp(-\psi (z_1) )  \right\} \subset \C^2$. 
Then the following two conditions are equivalent. 
\begin{enumerate}
\item 
$\partial \Omega $ is strictly pseudoconvex $($resp. strictly pseudoconcave$)$. 
\item
$\psi$ $($resp. $-\psi $$)$ is a strictly plurisubharmonic function.
\end{enumerate}
\end{proposition}

In the following, we construct a strictly pseudoconcave hypersurface $M_1$ which is a perturbation of 
the holomorphic open book $M$. 

\begin{proof}[Proof of Lemma \ref{main thm1}]
Let $f_1:[0,\rho_1^{-1}]\to \R_+$ and $f_2:[0,\rho_1^{-1}]\to \R_+$ be smooth positive functions satisfying the following conditions: 
\begin{enumerate}
\item $f_1(x)=c_1+\epsilon x^2$, $f_2(x)=c_2-\epsilon x^2$ near $x=0$ ($1<c_2<s\rho_1<c_1<\rho_2$); 
\item $g_1(x) =\log f_1(e^x)$ is strictly convex and $g_2(x) =\log f_2(e^x)$ is strictly concave; 
\item the derivatives satisfy $g_1'(-\log \rho_1)>0$ and $g_2'(-\log \rho_1)<-1$. 
\end{enumerate}
Recall that $(z_1, z_2)$ are the coordinates on $V=\Delta (1, \rho_2)\times \Delta (\rho_0^{-1}) \subset \C^2$. 
By Lemma \ref{V_2}, the hypersurfaces $H_1$ and $H_2$ in $V$ defined, respectively, by $|z_1|=f_1(|z_2|)$ and $|z_1|=f_2(|z_2|)$ are solid tori with the standard contact structures induced by the complex tangencies. 
Now we retake the coordinates $(w_1, w_2)$ on $V'$ so that $(w_1, w_2)=(z_1,z_2^{-1})$. 
Then, near $H_1$, the coordinates transformation between $V$ and $V'$ is $(w_1, w_2)=(z_1,z_2^{-1})$, 
and near $H_2$, it is $(w_1,w_2)=(z_1z_2, z_2^{-1})$ taking the identification $\psi : U'\to V'$ into account. 
Thus, we obtain the functions $h_1$ and $h_2$ satisfying
\begin{align*}
|z_1| &= f_1(|z_2|) \Leftrightarrow |w_2|=h_1(|w_1|),\quad \text{with } (w_1,w_2)=(z_1,z_2^{-1}),\\ 
|z_1| &= f_2(|z_2|) \Leftrightarrow |w_2|=h_2(|w_1|),\quad \text{with } (w_1,w_2)=(z_1z_2, z_2^{-1}). 
\end{align*}
Since the hypersurfaces $H_1$ and $H_2$ are strictly pseudoconcave, the functions $\log h_1(e^x)$ and $\log h_2(e^x)$ are strictly convex by Lemma~\ref{V_2}. Moreover, by the conditions $g_1'(-\log \rho_1)>0$ and $g_2'(-\log \rho_1)<-1$, 
their differential coefficients at $|w_2|=\rho_1$ are negative and positive, respectively.  
This is verified by the following argument. 

As in the proof of Lemma~\ref{V_2}, the differential coefficient $g_1'(-\log \rho_1)$ (resp. $g_2'(-\log \rho_1)$) represents the slope at $r_1=\rho_1^{-1}$ of the linear characteristic foliation on the $(\theta_1\theta_2)$-torus of the contact structure on $H_1$ (resp. $H_2$). 
Now we put $\phi _1=\arg w_1$ and $\phi _2=\arg w_2$. 
Then the coordinates transformation between $(\theta _1, \theta _2)$ and $(\phi_1, \phi_2)$ is given by 
\begin{align*}
\begin{pmatrix}\phi_1\\ \phi_2\end{pmatrix}
=\begin{pmatrix}1&0\\0&-1\end{pmatrix} 
\begin{pmatrix}
\theta_1 \\
\theta_2 
\end{pmatrix} \quad \text{near } H_1, \\ 
\begin{pmatrix}\phi_1\\ \phi_2\end{pmatrix}
=\begin{pmatrix}1&1\\0&-1\end{pmatrix}
\begin{pmatrix}
\theta_1 \\
\theta_2 
\end{pmatrix} \quad \text{near } H_2. 
\end{align*}
Since the differential coefficient of $\log h_1(e^x)$  (resp. $\log h_2(e^x)$) at $|w_2|=\rho_1$ is the slope of the linear characteristic foliation on the $(\phi_1\phi_2)$-torus, it is easily checked to be negative (resp. positive). 

Hence, there exists a smooth function $h$ connecting $h_1$ and $h_2$ such that $\log h(e^x)$ is strictly convex. 
Indeed, we can take a strictly convex function $\widetilde {h}$ connecting $\log h_1(e^x)$ and $\log h_2(e^x)$,
and then we can define $h ={\exp} \circ {\widetilde{h}} \circ {\exp ^{-1}}$.  
Then, the hy\-per\-sur\-face $S$ defined by $|w_2|=h(|w_1|)$ is $T^2\times I$ with the standard contact structure by Lemma \ref{V_2}. 
The union of $H_1$, $H_2$ and $S$ is a hypersurface diffeomorphic to the 3-sphere, which we   
denote by $M_1$. By construction, the complex tangencies define a negative contact structure on $M_1$. 
Thus, it is a strictly pseudoconcave hypersurface embedded in $C\subset E$. 

In the above construction, the strictly pseudoconcavity of $M_1$ is nothing but the conditions (1), (2), (3) and the strictly convexity of $\widetilde {h}$. Such functions $f_1$, $f_2$ and $\widetilde{h}$ can be flexibly chosen so that the corresponding family of strictly pseudoconcave 3-spheres foliates the complex domain $\Int C_1$. 

Then, the following lemma proves the existence of a strictly plurisubharmonic function, and this concludes the proof. 
\end{proof}

\begin{lemma}
Let $\gamma \colon X\to \R$ be a proper smooth regular function on a complex manifold $X$ 
such that the complex tangencies define a contact structure on the level sets $\gamma ^{-1}(c)$ for all $c\in \gamma (X)$. 
Then, there exists a smooth regular function $g \colon \gamma(X)\to \R$ such that $g\circ \gamma $ is strictly plurisubharmonic on $X$. 
\end{lemma}

\begin{proof}

We use the following notations:
$$d=\partial+\bar\partial, \;\; d^{\C}=i(\partial-\bar\partial ). $$
Notice that $d^{\C}\phi = (d\phi) \circ J$ for every function $\phi$. 
Since $d^{\C}(g\circ \gamma)= (g'\circ \gamma )\,d^{\C}\gamma $, we have 
$$-dd^{\C}(g\circ \gamma )=-d(g'(\gamma )d^{\C}\gamma )=-g''(\gamma )d\gamma \wedge d^{\C}\gamma -g'(\gamma )dd^{\C}\gamma .$$
The term $-d\gamma \wedge d^{\C}\gamma $ is positive on any complex line except on the complex tangency of the level set $\gamma ^{-1}(c)$, and zero on the complex tangency, 
because \begin{eqnarray*}
-d\gamma \wedge d^{\C}\gamma (u, Ju)=\dfrac{1}{2}(d\gamma (Ju)d^{\C}\gamma (u)-d\gamma (u)d^{\C}\gamma (Ju))\\
=\dfrac{1}{2}(d\gamma (Ju)d\gamma (Ju)-d\gamma (u)d\gamma (-u))=\dfrac{1}{2}((d\gamma (u))^2+(d\gamma (Ju))^2).
\end{eqnarray*}
On the other hand, the term $-dd^{\C}\gamma $ is positive on the complex tangency of $\gamma ^{-1}(c)$. 
Indeed, if $u, Ju\in T\gamma ^{-1}(c)$, then
\begin{eqnarray*}
-dd^{\C}\gamma (u, Ju)=\dfrac{1}{2}(d^{\C}\gamma ([u, Ju])-u(d^{\C}\gamma (Ju))+Ju(d^{\C}\gamma (u)))\\
=\dfrac{1}{2}(d\gamma (J[u, Ju])-u(d\gamma (-u))+Ju(d\gamma (Ju)))=\dfrac{1}{2}d\gamma (J[u, Ju])>0.
\end{eqnarray*}
Notice that the last inequality is the consequence of the contactness assumption. 
Hence, if we choose $g$ so that $g'>0$, $g''>0$ and $g''/g'$ is large enough on $\gamma (X)$, 
then $g\circ \gamma $ is strictly plurisubharmonic on $X$.  
\end{proof}

\begin{proof}[Proof of Main Theorem]
Endow $B^4$ with the complex structure $J$ induced by an o\-ri\-en\-ta\-tion-preserving diffeomorphism $B^4 \cong D_1$, the 4-ball in $E$ bounded by $M_1$. Then, $(B^4, J)$ is of Calabi-Eckmann type and with strictly pseudoconcave boundary $(S^3, \xi)$, where $\xi$ is the induced contact structure.

Since $J$ is homotopic, through almost complex structures, to the standard complex structure of $B^4 \subset \C^2$, the boundary contact structure $\xi$ is homotopic as a plane field to the standard positive tight contact structure of $S^3$.

We are left to show the compatibility of the contact structure on $M_1 \cong S^3$ with the open book decomposition inherited from $M$ by a suitable diffeomorphism $\phi \colon M \to M_1$ compatible with the splitting $M = (\Bd A \times B^2) \cup_\psi (A \times S^1)$ of the definition of $M$ in Section \ref{psh}, and the splitting $M_1 = H_1 \cup H_2 \cup S$ above, that is $\phi(\Bd A \times B^2) = H_1 \cup H_2$ and $\phi(A \times S^1) = S$. We want to prove that the contact form $\alpha$ is positive on the binding (oriented as the boundary of a page) and that $d\alpha$ is a volume form on the pages (oriented as holomorphic curves of $E$) of the open book decomposition, see \cite[Section 9.2]{OZ04}.

Since $u$ is strictly plurisubharmonic on $C_1$, the 1-form $\alpha =-d^{\C}u$ is a contact form on each level set of $u$, and the 2-form $d\alpha $ defines a symplectic structure compatible with the complex structure $J$. The contactness of $M_1$ is equivalent to the fact that the restriction $(\alpha \wedge d\alpha)_{|TM_1}$ is a volume form. On the other hand, the open book decomposition of $M_1$ is given by the function $$\arg z_2=\dfrac{z_2}{||z_2||}\colon M_1 - z_2^{-1}(0)\to S^1.$$ The vector $\frac{\partial}{\partial \theta _1}$ is tangent to the binding and the tangent space of the page is spanned by $\frac{\partial}{\partial \theta _1}$ and $V$, where $V$ is the tangent vector of the curve defined by $f_1$, $f_2$ and $h$. 

Now, we check the compatibility. 
It follows from the condition (1) of $f_1$ and $f_2$ that $$\alpha \left( \dfrac{\partial}{\partial \theta _1} \right)_{\!\! z_2=0}=-d^{\C}u \left(\dfrac{\partial}{\partial \theta _1}\right)_{\!\! z_2=0}=r_1 \left(\dfrac{\partial u}{\partial r _1}\right)_{\!\! z_2=0}\ne 0,$$ which implies 
the positivity of $\alpha $ along the binding. 

In order to see that $d\alpha $ is a volume form on the pages, it is enough to show that the vectors $\frac{\partial}{\partial \theta _1}$, $V$ and $R$ span the tangent space of $M_1$, where $$R=J \left(\dfrac{\grad u}{||\grad u ||}\right)$$ is the Reeb vector field of the contact form $\alpha_{| TM_1}$. 
Since the $r_2$ component of the gradient vector is positive except on the binding, so is the $\theta _2$ component of $R$. Therefore, the three vectors indeed span the tangent space except on the binding.  
\end{proof}






\section*{Acknowledgements}
The first author is supported by JSPS KAKENHI Grant Number JP17K14193. 
The second author acknowledges support of the 2013 ERC Advanced Research Grant 340258 TADMICAMT.
He is member of the group GNSAGA of Istituto Nazionale di Alta Matematica ``Francesco Severi'', Italy.

\let\emph\textsl


\begin{thebibliography}{10}
\bibitem{BEM}
M S Borman, Y Eliashberg, E Murphy,
\emph{Existence and classification of overtwisted contact structures in all dimensions}, Acta Math. Vol. 215, No. 2 (2015), 281--361.

\bibitem{CE53}
E Calabi, B Eckmann, \emph{A class of compact, complex manifolds which are not algebraic}, Ann. of Math. {\bf 58} (1953), 494--500. 



\bibitem{DKZ17}
A J Di Scala, N Kasuya, D Zuddas, \emph{Non-K\"{a}hler complex structures on $\R^4$}, Geometry \& Topology {\bf 21} (2017) 2461--2473. 

\bibitem{DKZ15b}
A J Di Scala, N Kasuya, D Zuddas, \emph{Non-K\"{a}hler complex structures on $\R^4$, ${\rm II}$}, J. Symplectic Geom. (to appear), also arXiv:1511.08471. 

\bibitem{El85}
Y Eliashberg, \emph{Complexification of contact structures on $3$-dimensional manifolds}, Uspekhi Mat. Nauk {\bf 40} (1985), 161--162 (in Russian). English translation: Russian Math. Surveys {\bf 40} (1985), 123--124.

\bibitem{El89}
Y Eliashberg, \emph{Classification of overtwisted contact structures on $3$-manifolds}, 
Invent. Math. {\bf 98} (1989), no. 3, 623--637. 

\bibitem{EH02}
J Etnyre, K Honda, \emph{On symplectic cobordism}, Math. Ann. {\bf 323} (2002), 31--39.

\bibitem{Ga02}
D T Gay, \emph{Explicit concave fillings of contact three-manifolds}, Math. Proc. Camb. Phil. Soc. {\bf 133} (2002), 431--441. 



\bibitem{Ha82}
J Harer, \emph{How to construct all fibered knots and links}, 
Topology {\bf 21} (1982), 263--280. 

\bibitem{Hi64}
H. Hironaka, \emph{Resolution of singularities of an algebraic variety over a field of characteristic zero I, II}, 
Ann. of Math. {\bf 79} (1964), 109--326. 

\bibitem{Kod63}
K Kodaira, \emph{On Compact Analytic Surfaces: II}, Ann. of Math. Vol. 77, No. 3 (1963), 563--626.

\bibitem{LM97}
P Lisca, G Mati\'c,
\emph{Tight contact structures and Seiberg-Witten invariants}, 
Invent. Math. {\bf 129} (1997), 509--525. 

\bibitem{Mat82}
Y Matsumoto, \emph{On $4$-manifolds fibered by tori}, Proc. Japan Acad. Ser. A Math. Sci. {\bf 58} (1982), no. 7, 298--301.

\bibitem{OZ04}
B Ozbagci, A I Stipsicz, {Surgery on contact 3-manifolds and Stein surfaces}, Bolyai Society Mathematical Studies, 13, Springer-Verlag, 2004.

\bibitem{Ro90}
D. Rolfsen, 
{\sl Knots and links},
Corrected reprint of the 1976 original, Mathematics Lecture Series, 7. Publish or Perish, Inc., Houston, 1990.

\bibitem{Ro65}
H Rossi,
\emph{Attaching analytic spaces to an analytic space along a pseudoconcave boundary}, 
Proc. Conf. Complex Analysis (Minneapolis, 1964), Springer-Verlag (1965), 242--256.

\bibitem{TW75}
W P Thurston, H E Winkelnkemper,
\emph{On the existence of contact forms}, 
Proc. Amer. Math. Soc. {\bf 52} (1975), 345--347. 

\end{thebibliography}
\end{document}